\documentclass[11pt]{article}

\usepackage{amsmath, amssymb, amsthm}
\usepackage{fullpage}
\usepackage{graphicx}
\usepackage{hyperref}
\usepackage{mathrsfs}
\usepackage{float}
\usepackage[ruled,vlined]{algorithm2e}
\usepackage{placeins}
\usepackage{color}
\usepackage{xcolor}
\newtheorem{proposition}{Proposition}

\newtheorem{theorem}{Theorem}[section]
\newtheorem{lemma}[theorem]{Lemma}
\newtheorem{corollary}[theorem]{Corollary}
\newtheorem{remark}[theorem]{Remark}

\date{}

\title{\bf Nearest Graph Laplacians with Prescribed Connected Components: A Convex Framework for Network Reconstruction}

\begin{document}

	\author{  
		Udit Raj\thanks{Department of Electrical Engineering and Computer Science, Indian Institute of Science Education and Research, Bhopal, India. Email: \texttt{ur376@snu.edu.in; udit2026@iiserb.ac.in}} 
		\and
		Sudeepto Bhattacharya\thanks{Department of Mathematics, Shiv Nadar Institution of Eminence, Delhi-NCR Email: \texttt{sudeepto.bhattacharya@snu.edu.in}}
        \and
        Prince Kanhya\thanks{Corresponding Author: Department of Mathematics, Indian Institute of Technology Guwahati, Assam, India. Email: \texttt{pkanhya77@gmail.com}}
        }
	\date{\today}
	\maketitle
\begin{abstract}
	We study the problem of constructing the nearest graph Laplacian matrix to a given Laplacian while enforcing a prescribed connected-component structure. Let the vertex set be partitioned into nonempty disjoint blocks $C_1,\ldots,C_k$, and let $U=[u_1,\ldots,u_k]$ be the matrix of the corresponding block-indicator vectors. The constraint $MU=0$ ensures that these prescribed indicators lie in the nullspace of the optimized Laplacian $M^\star$, and hence the associated graph has at least $k$ connected components. To guarantee exactly the prescribed components, we impose additional block-connectivity constraints on the principal blocks $M_j=M[C_j,C_j]$. These constraints ensure that each prescribed block induces a connected weighted subgraph. The resulting problem is a convex semidefinite optimization problem with a strictly convex Frobenius-norm objective. We prove existence and uniqueness of the minimizer and show that the optimized Laplacian has exactly the prescribed connected components, with nullspace $\operatorname{span}\{u_1,\ldots,u_k\}$. The framework proposed in this work provides a principled tool for quantifying the minimum structural intervention required to transform a graph-based network into one having a prescribed group-separated structure. Numerical examples, including the Sampson monastery positive-affection network, illustrate the nearest faction-consistent weighted reconstruction and the minimum Laplacian perturbation required to realize the prescribed faction structure.
\end{abstract}

\vspace{0.5cm}


\vspace{0.5cm}

\newpage

	\section{Introduction}
	
Social networks frequently exhibit group or community structure, in which 
individuals are more strongly connected within groups than across groups 
\cite{ArefEtAl2020, GirvanNewman2002}. In empirical applications, however, the grouping 
suggested by observed interactions may not coincide exactly with externally 
known social, institutional, or demographic classifications 
\cite{TraudKelsicMuchaPorter2011}. Moreover, observed network data may contain 
measurement errors, missing relationships, uncertain edges, or noise, making 
the recorded network only an imperfect representation of the underlying social 
organization \cite{MartinBallNewman2016,Newman2018}. These considerations 
motivate the problem of reconstructing a network that is consistent with a 
prescribed group structure while retaining as much of the observed interaction 
pattern as possible.

    
Graph Laplacian matrices are fundamental objects in spectral graph theory, network analysis, consensus dynamics, and graph-based data science \cite{Olfati-Saber}. In network control and opinion dynamics, Laplacian-based models play a central role in consensus, prescribed
group consensus, and opinion-clustering problems \cite{Spielman2010, TT}. For an undirected weighted graph, the nullspace of its Laplacian matrix carries direct information about the connected-component structure of the graph. In particular, the dimension of the nullspace of a graph Laplacian is equal to the number of connected components of the associated graph. Thus, controlling the nullspace of a Laplacian provides an algebraic way to control the component structure of a graph.

Several related classes of graph-learning and matrix-nearness
problems have been studied in the literature. Sato considered the
construction of a nearest graph Laplacian to a given matrix under a
known network structure, while Sato and Suzuki subsequently developed
specialized algorithms for the nearest graph-Laplacian problem in the
Frobenius norm \cite{Sato2018,SatoSuzuki2024}. 
More broadly, structured
matrix-nearness problems under prescribed spectral and structural
constraints have also been investigated \cite{PKUD}. For general background on matrix perturbation theory and the
sensitivity of spectral information under matrix modifications, we
refer to \cite{Chung1997}. 
Graph-learning methods
have also been proposed for obtaining graphs with a prescribed number
of connected components. In particular, the constrained Laplacian-rank
approach learns a similarity graph with exactly $k$ connected
components by imposing a rank condition on its Laplacian
\cite{NieEtAl2016}. Related approaches infer graphs with monotone
topological properties and multiple connected components
\cite{PavezEgilmezOrtega2018}, while graph-Laplacian estimation under
connectivity and sparsity constraints has been investigated from a
statistical graph-learning perspective
\cite{EgilmezPavezOrtega2017}.

In this work, the problem addressed, unlike the above approaches, does not infer an unknown graph topology or prescribe only the number of connected components. Instead, both the original Laplacian $L$ and the exact vertex partitions are given well in advance. In this work, we study the following nearest-Laplacian problem. Given a graph Laplacian matrix $L$ and a prescribed partition of the vertex set into nonempty pairwise disjoint blocks
$C_1,\ldots,C_k$, that is,
	\[
	V=C_1\cup C_2\cup\cdots\cup C_k,
	\qquad
	C_i\cap C_j=\emptyset \quad (i\ne j),
	\]
	we seek a graph Laplacian matrix that is nearest to $L$ in the Frobenius norm and whose associated graph has exactly the connected components $C_1,\ldots,C_k$. In other words, the goal is to minimally modify the original Laplacian while ensuring that the optimized graph has the desired component structure.
    
 {A prescribed component structure can, of course, be obtained by manually
	removing selected inter-component edges. However, such a construction is not
	unique and may unnecessarily distort the original graph. The purpose of the
	present work is to construct the desired component structure while changing
	the original Laplacian as little as possible. Thus, the proposed method should
	be viewed as a minimal-perturbation correction of the original graph rather
	than an arbitrary redesign.}
    
	The prescribed components are encoded using their indicator vectors. For each block $C_j$, let $u_j$ denote the indicator vector of $C_j$, and define
	$$
	U=[u_1,u_2,\ldots,u_k].
	$$
	The nullspace constraint
	$$
	MU=0
	$$
	ensures that all prescribed block-indicator vectors belong to the nullspace of the optimized Laplacian. Since the vectors $u_1,\ldots,u_k$ are linearly independent, this implies
	$$
	\dim\ker(M)\ge k.
	$$
	Hence, the graph associated with $M$ has at least $k$ connected components \cite{Mohar1991}.
	
	However, the constraint $MU=0$ alone does not guarantee that the optimized graph has exactly $k$ connected components. A prescribed block $C_j$ may itself split into two or more disconnected components, creating additional nullspace directions. Therefore, to guarantee exactly the prescribed components, it is necessary to prevent extra zero directions inside each block.
	
	We achieve this by imposing block-connectivity constraints. Let
	$$
	M_j=M[C_j,C_j]
	$$
	be the principal block of $M$ corresponding to the prescribed block $C_j$. For each non-singleton block, we require $M_j$ to be positive definite on the subspace orthogonal to the all-ones vector. Equivalently, we impose a positive lower bound on the algebraic connectivity of each prescribed block. This ensures that every prescribed block induces a connected weighted subgraph.
	
	The resulting problem is a convex semidefinite optimization problem with a strictly convex Frobenius-norm objective. The objective
	$$
	\|M-L\|_F^2
	$$
	keeps the optimized Laplacian close to the original one, while the nullspace and block-connectivity constraints enforce the desired graph structure. Since the feasible set is nonempty, closed, and convex, and the objective is strictly convex, the problem admits a unique global minimizer.
	
	This formulation is useful in settings where one wants to preserve a given network as much as possible while enforcing a desired component or cluster structure. For example, in continuous-time consensus dynamics
	$$
	\dot{x}(t)=-Mx(t),
	$$
	the nullspace of $M$ determines the limiting consensus behavior. If $M$ has a one-dimensional nullspace spanned by the all-ones vector, then all agents reach global consensus. If $\dim\ker(M)=k$, then the system reaches cluster consensus over $k$ connected components. Thus, the proposed nearest-Laplacian construction provides a minimal-perturbation approach for designing networks with prescribed cluster-consensus structure.
   
    The proposed nearest-Laplacian framework provides such a reconstruction by computing the unique weighted Laplacian with exactly the prescribed components that minimizes the Frobenius-norm perturbation from the original Laplacian. The resulting distance quantifies the minimum Laplacian intervention required to make the observed network consistent with the prescribed social organization, while the algebraic connectivity of each optimized block measures the internal connectivity of the reconstructed groups. We would also like to put a caveat about the proposed reconstruction imposes complete separation between prescribed groups and should therefore be interpreted as a minimum-perturbation reference network consistent with exact faction separation, rather than as a claim that empirical social communities generally contain no inter-group relationships.
	
	The main contributions of this paper are as follows:
	\begin{itemize}
		\item We formulate a nearest-Laplacian problem for constructing a graph Laplacian with prescribed connected components.
		\item We show that the constraint $MU=0$ guarantees at least $k$ connected components, but not necessarily exactly $k$ components.
		\item We introduce block-connectivity semidefinite constraints that guarantee exactly the prescribed components.
		\item We prove that the resulting convex semidefinite optimization problem has a unique minimizer.
		\item We provide a blockwise semidefinite implementation, numerical verification procedures, and applications to synthetic and real social-network data.
	\end{itemize}
	
\section{Notation}

Let $G=(V,E,W)$ be a weighted undirected graph with vertex set
\[
V=\{1,2,\ldots,n\}.
\]
The matrix $W=[w_{ij}]$ denotes the symmetric nonnegative weight matrix, where $w_{ij}\ge 0$ for $i\ne j$ and $w_{ii}=0$. The graph Laplacian associated with $W$ is
\[
L=D-W,
\]
where
\[
D=\operatorname{diag}(d_1,\ldots,d_n),
\qquad
d_i=\sum_{j=1}^{n}w_{ij}.
\]

\begin{remark}
	If the original graph is unweighted, then its adjacency matrix $A$ may be viewed as a special case of a weighted adjacency matrix $W$, with entries in $\{0,1\}$. The input Laplacian is then
	\[
	L=D_A-A,
	\]
	where $D_A$ is the diagonal degree matrix associated with $A$. Since the optimization variables are continuous, the optimized Laplacian $M^\star$ generally corresponds to a weighted graph with weight matrix $W^\star$, where
	\[
	W^\star_{ij}=-M^\star_{ij},\qquad i\ne j,
	\]
	and
	\[
	W^\star_{ii}=0.
	\]
\end{remark}
Throughout the paper, $L\in\mathbb R^{n\times n}$ denotes the given original graph Laplacian, $M\in\mathbb{R}^{n\times n}$ denotes a candidate
Laplacian and the optimization variable  and $M^\star\in\mathbb R^{n\times n}$ denotes the optimized graph Laplacian.

In later sections, $\succeq$ denotes the Loewner partial order on symmetric
matrices; namely, for symmetric matrices $A$ and $B$,
\[
A\succeq B
\]
means that $A-B$ is positive semidefinite
\cite{HornJohnson2013}.

Let the prescribed partition of the vertex set be
\[
V=C_1\cup C_2\cup\cdots\cup C_k,
\qquad
C_i\cap C_j=\emptyset \quad (i\ne j),
\]
where each $C_j$ is nonempty. For each block $C_j$, define its indicator vector $u_j\in\mathbb R^n$ by
\[
(u_j)_i=
\begin{cases}
	1, & i\in C_j,\\
	0, & i\notin C_j.
\end{cases}
\]
The block-indicator matrix is defined by
\[
U=[u_1,u_2,\ldots,u_k]\in\mathbb R^{n\times k}.
\]
Since $C_1,\ldots,C_k$ form a partition of $V$, we have
\[
\mathbf 1_n=u_1+u_2+\cdots+u_k.
\]

For each $j=1,\ldots,k$, let
\[
n_j=|C_j|
\]
and let
\[
M_j=M[C_j,C_j]\in\mathbb R^{n_j\times n_j}
\]
denote the principal submatrix of $M$ corresponding to the vertices in $C_j$.

Define the orthogonal projector onto the subspace orthogonal to $\mathbf 1_{n_j}$ by
\[
P_j
=
I_{n_j}
-
\frac{1}{n_j}\mathbf 1_{n_j}\mathbf 1_{n_j}^T.
\]
Thus,
\[
P_j\mathbf 1_{n_j}=0,
\]
and
\[
P_jx=x
\quad\text{whenever}\quad
x\perp \mathbf 1_{n_j}.
\]

For numerical implementation, we also use a matrix
\[
Q_j\in\mathbb R^{n_j\times(n_j-1)}
\]
whose columns form an orthonormal basis of the subspace
\[
\mathbf 1_{n_j}^{\perp}
=
\{x\in\mathbb R^{n_j}:x^T\mathbf 1_{n_j}=0\}.
\]
Thus,
\[
Q_j^T\mathbf 1_{n_j}=0,
\qquad
Q_j^TQ_j=I_{n_j-1}.
\]

We write $\ker(M)$ for the nullspace of $M$. The eigenvalues of a symmetric matrix $M$ are ordered as
\[
\lambda_1(M)\le \lambda_2(M)\le\cdots\le \lambda_n(M).
\]
	
	\section{Preliminaries}
	
	We first recall a standard property of weighted graph Laplacians. This result connects the nullspace of a Laplacian matrix with the connected components of the associated graph \cite{Chung1997, Biggs1993, Bapat2010, Merris1994}.
	
	\begin{lemma}[Nullspace of a weighted graph Laplacian]
		Let $H\in\mathbb R^{m\times m}$ satisfy
		$$
		H=H^T,\qquad H\mathbf 1_m=0,\qquad H_{pq}\le 0
		\quad (p\ne q).
		$$
		Define the edge weights
		$$
		w_{pq}=-H_{pq}\ge0,\qquad p\ne q.
		$$
		Then, for every $x\in\mathbb R^m$,
		$$
		x^THx
		=
		\frac12\sum_{p,q=1}^{m}w_{pq}(x_p-x_q)^2.
		$$
		Consequently, $H\succeq0$, and
		$$
		x\in\ker(H)
		$$
		if and only if $x$ is constant on every connected component of the associated weighted graph. Therefore,
		$$
		\dim\ker(H)
		=
		\text{number of connected components of the associated graph}.
		$$
		In particular, the associated graph is connected if and only if
		$$
		\ker(H)=\operatorname{span}\{\mathbf 1_m\},
		$$
		or equivalently,
		$$
		\lambda_2(H)>0.
		$$
	\end{lemma}
	
	\begin{proof}
		Since $H\mathbf 1_m=0$, we have
		$$
		H_{pp}
		=
		-\sum_{q\ne p}H_{pq}
		=
		\sum_{q\ne p}w_{pq}.
		$$
		Thus,
		$$
		\begin{aligned}
			x^THx
			&=
			\sum_{p=1}^{m}H_{pp}x_p^2
			+
			\sum_{p\ne q}H_{pq}x_px_q\\
			&=
			\sum_{p=1}^{m}
			\left(\sum_{q\ne p}w_{pq}\right)x_p^2
			-
			\sum_{p\ne q}w_{pq}x_px_q\\
			&=
			\frac12\sum_{p,q=1}^{m}
			w_{pq}(x_p-x_q)^2.
		\end{aligned}
		$$
		Hence $x^THx\ge0$ for every $x\in\mathbb R^m$, and therefore $H\succeq0$.
		
		Moreover,
		$$
		x^THx=0
		$$
		if and only if
		$$
		x_p=x_q
		$$
		for every edge of positive weight. Therefore, $x$ must be constant on each connected component of the associated graph. Conversely, any vector that is constant on each connected component satisfies $Hx=0$.
		
		Hence the nullspace of $H$ is spanned by the indicator vectors of the connected components, and its dimension equals the number of connected components. The final equivalences follow immediately.
	\end{proof}
	
	\begin{remark}
		The preceding lemma shows that, for a graph Laplacian, controlling the nullspace is equivalent to controlling the connected-component structure. However, to obtain exactly $k$ connected components, it is not enough to place $k$ independent prescribed vectors in the nullspace; one must also ensure that no additional independent nullspace directions are created.
	\end{remark}

	\section{Problem Formulation}
	
	Let $L\in\mathbb R^{n\times n}$ be the given graph Laplacian matrix, and let
	\[
	V=C_1\cup C_2\cup\cdots\cup C_k,
	\qquad
	C_i\cap C_j=\emptyset\quad (i\ne j),
	\]
	be the prescribed partition of the vertex set into nonempty blocks. For each block
	$C_j$, let $u_j$ be its indicator vector, and define
	\[
	U=[u_1,u_2,\ldots,u_k]\in\mathbb R^{n\times k}.
	\]
	
	{We want to construct a graph Laplacian matrix $M$ that is closest to $L$ in the
	Frobenius norm and whose associated graph has exactly the prescribed connected
	components $C_1,\ldots,C_k$.} The minimal-perturbation objective is
	\[
	\|M-L\|_F^2.
	\]
	This objective ensures that, among all Laplacian matrices satisfying the desired
	component constraints, we choose the one nearest to the original Laplacian and unique closest such matrix is denoted by $M^\star$.
	
	A natural first requirement is that the prescribed block-indicator vectors lie in
	the nullspace of $M$. This is imposed by the linear constraint
	\[
	MU=0.
	\]
	Since the blocks form a partition of the vertex set, we have
	\[
	\mathbf 1_n=u_1+u_2+\cdots+u_k.
	\]
	Hence $MU=0$ automatically implies
	\[
	M\mathbf 1_n=0.
	\]
	Thus, the row-sum condition of a graph Laplacian is already included in the
	constraint $MU=0$.
	
	Together with symmetry and nonpositive off-diagonal entries, the condition
	$MU=0$ gives a weighted graph Laplacian whose nullspace contains the prescribed
	block-indicator vectors. However, this condition alone only implies
	\[
	\operatorname{span}\{u_1,\ldots,u_k\}\subseteq\ker(M),
	\]
	and hence
	\[
	\dim\ker(M)\ge k.
	\]
	Therefore, the associated graph has at least $k$ connected components. It may
	have more than $k$ components if some prescribed block $C_j$ splits into two or
	more disconnected components.
	
	To prevent such additional splitting, we impose a block-connectivity constraint.
	For each $j=1,\ldots,k$, let
	\[
	n_j=|C_j|,
	\qquad
	M_j=M[C_j,C_j]\in\mathbb R^{n_j\times n_j},
	\]
	and define the orthogonal projector
	\[
	P_j
	=
	I_{n_j}
	-
	\frac{1}{n_j}\mathbf 1_{n_j}\mathbf 1_{n_j}^T.
	\]
	The matrix $P_j$ is the orthogonal projector onto the subspace
	\[
	\mathbf 1_{n_j}^{\perp}
	=
	\{x\in\mathbb R^{n_j}:x^T\mathbf 1_{n_j}=0\}.
	\]
	
	For a fixed parameter $\varepsilon>0$, we consider the strengthened nearest
	Laplacian problem
	\[
	\begin{aligned}
		\min_{M\in\mathbb R^{n\times n}}\quad
		&\|M-L\|_F^2\\
		\text{subject to}\quad
		&MU=0,\\
		&M=M^T,\\
		&M_{ij}\le0,\qquad i\ne j,\\
		&M_j\succeq \varepsilon P_j,\qquad j=1,\ldots,k.
	\end{aligned}
	\tag{P}
	\]
	The constraint
	\[
	M_j\succeq\varepsilon P_j
	\] 
	means that
	\[
	x^TM_jx\ge \varepsilon x^TP_jx
	\qquad
	\text{for every }x\in\mathbb R^{n_j}.
	\]
    
	In particular, if $x\perp\mathbf 1_{n_j}$, then $P_jx=x$, and therefore
	\[
	x^TM_jx\ge \varepsilon\|x\|_2^2.
	\]
	Thus, for every non-singleton block $C_j$, the block $M_j$ is positive definite
	on $\mathbf 1_{n_j}^{\perp}$, and hence
	\[
	\lambda_2(M_j)\ge\varepsilon.
	\]
	This guarantees that each prescribed block induces a connected weighted
	subgraph. Singleton blocks are treated as connected by convention.
	
	The problem $(P)$ is a convex semidefinite optimization problem {\cite{BoydVandenberghe2004}}. The objective is
	a strictly convex quadratic function of $M$. The constraints $MU=0$ and $M=M^T$
	are linear equality constraints and hence affine, the inequalities $M_{ij}\le0$ are linear, and the constraints
	\[
	M_j-\varepsilon P_j\succeq0
	\]
	are linear matrix inequalities.
	
	\begin{proposition}[Feasibility of the constraint set]
		The feasible set of problem $(P)$ is nonempty.
	\end{proposition}
	
	\begin{proof}
		We construct one feasible matrix explicitly. First arrange the vertices
		blockwise according to the prescribed order
		\[
		C_1,C_2,\ldots,C_k.
		\]
		In this blockwise ordering, define
		\[
		\widehat M
		=
		\operatorname{diag}(\varepsilon P_1,\varepsilon P_2,\ldots,\varepsilon P_k).
		\]
		Then transform $\widehat M$ back to the original vertex ordering.
		
		Each block $\varepsilon P_j$ is symmetric and satisfies
		\[
		\varepsilon P_j\mathbf 1_{n_j}=0.
		\]
		Moreover, for $p\ne q$,
		\[
		(\varepsilon P_j)_{pq}
		=
		-\frac{\varepsilon}{n_j}
		\le0.
		\]
		Also,
		\[
		\varepsilon P_j\succeq\varepsilon P_j.
		\]
		Since the matrix is block diagonal with respect to the prescribed partition, it
		satisfies
		\[
		MU=0.
		\]
		Therefore, the constructed matrix satisfies all constraints of $(P)$, and the
		feasible set is nonempty.
	\end{proof}
	
	\begin{proposition}[Existence and uniqueness of the minimizer]
		Problem $(P)$ admits a unique global minimizer.
	\end{proposition}
	
	\begin{proof}
		The feasible set of $(P)$ is nonempty by the previous proposition. It is closed
		and convex, since it is defined by affine equalities, linear inequalities, and
		linear matrix inequalities.
		
		The objective function
		\[
		f(M)=\|M-L\|_F^2
		\]
		is continuous, coercive, and strictly convex in $M$. Indeed, in vectorized
		coordinates,
		\[
		f(M)
		=
		\|\operatorname{vec}(M)-\operatorname{vec}(L)\|_2^2,
		\]
		whose Hessian is $2I$.
		
		Since the feasible set is nonempty and closed, and since the objective is
		continuous and coercive, a minimizer exists. Moreover, the objective is strictly
		convex and the feasible set is convex, the minimizer is unique.
	\end{proof}
	
	\begin{remark}
		The formulation $(P)$ is stronger than the formulation containing only the
		constraint $MU=0$. The constraint $MU=0$ guarantees that the optimized graph has
		at least $k$ connected components. The additional semidefinite constraints
		\[
		M_j\succeq \varepsilon P_j,\qquad j=1,\ldots,k,
		\]
		prevent any prescribed block from splitting into smaller connected components.
		Therefore, the strengthened formulation is designed to produce exactly the
		prescribed components.
	\end{remark}


	\section{The Nullspace Constraint and the At-Least-k Component Property}
	
	We first record what is guaranteed by the nullspace constraint
	\[
	MU=0
	\]
	alone. This constraint ensures that the prescribed block-indicator vectors belong to the nullspace of $M$. However, it does not by itself guarantee that the optimized graph has exactly $k$ connected components.
	
	\begin{theorem}[Nullspace constraint and prescribed block indicators]
	Let $M\in\mathbb R^{n\times n}$ satisfy
	\[
	MU=0,
	\]
	where
	\[
	U=[u_1,u_2,\ldots,u_k]
	\]
	is the block-indicator matrix associated with the prescribed partition
	\[
	V=C_1\cup C_2\cup\cdots\cup C_k,
	\qquad
	C_i\cap C_j=\emptyset\quad (i\ne j).
	\]
	Then
	\[
	\operatorname{span}\{u_1,u_2,\ldots,u_k\}\subseteq\ker(M).
	\]
	Consequently,
	\[
	\dim\ker(M)\ge k.
	\]
	If, in addition, $M$ is a graph Laplacian, then the graph associated with $M$ has at least $k$ connected components.
	\end{theorem}
	
	\begin{proof}
	Since
	\[
	U=[u_1,u_2,\ldots,u_k],
	\]
	the constraint
	\[
	MU=0
	\]
	is equivalent to
	\[
	Mu_j=0,\qquad j=1,2,\ldots,k.
	\]
	Thus,
	\[
	u_j\in\ker(M),\qquad j=1,2,\ldots,k.
	\]
	
	Now let
	\[
	z\in\operatorname{span}\{u_1,u_2,\ldots,u_k\}.
	\]
	Then there exist scalars $\alpha_1,\alpha_2,\ldots,\alpha_k$ such that
	\[
	z=\sum_{j=1}^{k}\alpha_j u_j.
	\]
	Applying $M$, we obtain
	\[
	Mz
	=
	M\left(\sum_{j=1}^{k}\alpha_j u_j\right)
	=
	\sum_{j=1}^{k}\alpha_j Mu_j
	=
	0.
	\]
	Hence
	\[
	z\in\ker(M),
	\]
	and therefore
	\[
	\operatorname{span}\{u_1,u_2,\ldots,u_k\}\subseteq\ker(M).
	\]
	
	Since the blocks $C_1,\ldots,C_k$ are nonempty and pairwise disjoint, their indicator vectors
	\[
	u_1,u_2,\ldots,u_k
	\]
	are linearly independent. Therefore,
	\[
	\dim\operatorname{span}\{u_1,u_2,\ldots,u_k\}=k.
	\]
	Since this $k$-dimensional subspace is contained in $\ker(M)$, we obtain
	\[
	\dim\ker(M)\ge k.
	\]
	
	If $M$ is a graph Laplacian, then by the standard Laplacian nullity theorem, the nullity of $M$ equals the number of connected components of the associated graph. Hence the graph associated with $M$ has at least $k$ connected components.
	\end{proof}
	
	\begin{remark}
	The preceding theorem shows that the constraint $MU=0$ is sufficient to place the prescribed block-indicator vectors in the nullspace of $M$. However, it does not exclude additional nullspace directions. Such additional nullspace directions arise precisely when one or more prescribed blocks split into smaller disconnected components. Therefore, the constraint $MU=0$ guarantees at least $k$ connected components, but not necessarily exactly $k$ connected components.
	\end{remark}
	
	\begin{remark}
	To obtain exactly $k$ connected components, one needs the stronger condition
	\[
	\dim\ker(M)=k.
	\]
	Equivalently, if the eigenvalues of $M$ are ordered as
	\[
	0=\lambda_1(M)=\lambda_2(M)=\cdots=\lambda_k(M)<\lambda_{k+1}(M),
	\]
	then the associated graph has exactly $k$ connected components. In the strengthened formulation $(P)$, this is achieved by imposing the block-connectivity constraints
	\[
	M_j\succeq \varepsilon P_j,\qquad j=1,\ldots,k.
	\]
	\end{remark}

\section{Main Theorem: Exactly Prescribed Components}

We now show that the strengthened problem $(P)$ produces a Laplacian matrix whose connected components are exactly the prescribed blocks
\[
C_1,\ldots,C_k.
\]

\begin{theorem}[Nearest Laplacian with exactly prescribed components]
	Let $L\in\mathbb R^{n\times n}$ be a given graph Laplacian, and let
	\[
	V=C_1\cup C_2\cup\cdots\cup C_k,
	\qquad
	C_i\cap C_j=\emptyset\quad (i\ne j),
	\]
	be a prescribed partition of the vertex set into nonempty blocks. For each block $C_j$, let $u_j$ be its indicator vector, and define
	\[
	U=[u_1,u_2,\ldots,u_k].
	\]
	Let $M^\star$ be the unique minimizer of problem $(P)$.
	
	Then $M^\star$ is a weighted graph Laplacian whose associated graph has exactly the prescribed connected components
	\[
	C_1,\ldots,C_k.
	\]
	Moreover,
	\[
	\ker(M^\star)=\operatorname{span}\{u_1,u_2,\ldots,u_k\}.
	\]
	If $k<n$, then
	\[
	\lambda_{k+1}(M^\star)\ge \varepsilon>0.
	\]
\end{theorem}

\begin{proof}
	Since $M^\star$ is feasible for $(P)$, it satisfies
	\[
	M^\star U=0.
	\]
	As
	\[
	U=[u_1,u_2,\ldots,u_k],
	\]
	this is equivalent to
	\[
	M^\star u_j=0,\qquad j=1,\ldots,k.
	\]
	Therefore,
	\[
	u_j\in\ker(M^\star),\qquad j=1,\ldots,k.
	\]
	
	Since the prescribed blocks form a partition of the vertex set,
	\[
	\mathbf 1_n=u_1+u_2+\cdots+u_k.
	\]
	Thus,
	\[
	M^\star\mathbf 1_n
	=
	\sum_{j=1}^{k}M^\star u_j
	=
	0.
	\]
	Together with the constraints
	\[
	M^\star=(M^\star)^T,
	\qquad
	M^\star_{ij}\le0\quad (i\ne j),
	\]
	this shows that $M^\star$ is a weighted graph Laplacian.
	
	We next show that there are no edges between distinct prescribed blocks. Fix a block $C_j$ and a vertex $i\in C_j$. Since
	\[
	M^\star u_j=0,
	\]
	the $i$-th component gives
	\[
	\sum_{\ell\in C_j}M^\star_{i\ell}=0.
	\]
	On the other hand, since
	\[
	M^\star\mathbf 1_n=0,
	\]
	we also have
	\[
	\sum_{\ell=1}^{n}M^\star_{i\ell}=0.
	\]
	Subtracting these two equalities yields
	\[
	\sum_{\ell\notin C_j}M^\star_{i\ell}=0.
	\]
	For every $\ell\notin C_j$, we have $i\ne \ell$, and hence
	\[
	M^\star_{i\ell}\le0.
	\]
	A sum of nonpositive numbers can be zero only if every term is zero. Therefore,
	\[
	M^\star_{i\ell}=0,
	\qquad
	i\in C_j,\quad \ell\notin C_j.
	\]
	Thus, there are no edges between distinct prescribed blocks.
	
	Consequently, after a suitable permutation of the vertices, $M^\star$ becomes block diagonal:
	\[
	Q M^\star Q^T
	=
	\operatorname{diag}(M^\star_1,M^\star_2,\ldots,M^\star_k),
	\]
	where
	\[
	M^\star_j=M^\star[C_j,C_j].
	\]
	
	We now prove that each prescribed block is internally connected. Since there are no entries connecting $C_j$ with the other blocks and
	\[
	M^\star u_j=0,
	\]
	the corresponding principal block satisfies
	\[
	M^\star_j\mathbf 1_{n_j}=0.
	\]
	Thus, $\mathbf 1_{n_j}$ is always a zero direction of $M^\star_j$.
	
	From the constraint in $(P)$, we have
	\[
	M^\star_j\succeq \varepsilon P_j.
	\]
	Therefore, for every $x\in\mathbb R^{n_j}$,
	\[
	x^TM^\star_jx\ge \varepsilon x^TP_jx.
	\]
	If $x\perp \mathbf 1_{n_j}$, then $P_jx=x$, and hence
	\[
	x^TM^\star_jx\ge \varepsilon\|x\|_2^2.
	\]
	Thus, for every non-singleton block $C_j$, the matrix $M^\star_j$ is positive definite on
	\[
	\mathbf 1_{n_j}^{\perp}.
	\]
	It follows that
	\[
	\ker(M^\star_j)=\operatorname{span}\{\mathbf 1_{n_j}\},
	\]
	and hence
	\[
	\lambda_2(M^\star_j)\ge\varepsilon>0
	\]
	for every block with $n_j\ge2$. By the standard Laplacian nullity theorem, the subgraph induced by each non-singleton block $C_j$ is connected. Singleton blocks are connected by convention.
	
	Thus, there are no edges between distinct prescribed blocks, and each prescribed block is internally connected. Hence the connected components of the graph associated with $M^\star$ are exactly
	\[
	C_1,C_2,\ldots,C_k.
	\]
	
	Therefore, the graph has exactly $k$ connected components. By the Laplacian nullity theorem,
	\[
	\dim\ker(M^\star)=k.
	\]
	Since
	\[
	u_1,u_2,\ldots,u_k
	\]
	are linearly independent and belong to $\ker(M^\star)$, we conclude that
	\[
	\ker(M^\star)
	=
	\operatorname{span}\{u_1,u_2,\ldots,u_k\}.
	\]
	
	Finally, since $M^\star$ is permutation-similar to
	\[
	\operatorname{diag}(M^\star_1,M^\star_2,\ldots,M^\star_k),
	\]
	its eigenvalues are the union of the eigenvalues of the block matrices $M^\star_j$. Each block contributes exactly one zero eigenvalue. Hence the first $k$ eigenvalues of $M^\star$ are zero. If $k<n$, then at least one block is non-singleton, and
	\[
	\lambda_{k+1}(M^\star)
	=
	\min_{\{j:n_j\ge2\}}\lambda_2(M^\star_j)
	\ge
	\varepsilon.
	\]
	This completes the proof.
\end{proof}

\begin{corollary}
	Under the constraints of problem $(P)$, the connected components of the optimized graph are precisely the prescribed blocks
	\[
	C_1,\ldots,C_k.
	\]
\end{corollary}

\begin{proof}
	This follows directly from the preceding theorem.
\end{proof}
	
	\section{Algorithms}
	\begin{algorithm}[H]
		\caption{Blockwise Nearest-Laplacian Optimization}
		\label{alg:blockwise_optimization}
		\DontPrintSemicolon
		\footnotesize
		
		\KwData{Original Laplacian $L$, prescribed partition $C_1,\ldots,C_k$, parameter $\varepsilon>0$.}
		\KwResult{Optimized blocks $M_1^\star,\ldots,M_k^\star$.}
		
		Construct the block-indicator vectors $u_j$ and set
		\[
		U=[u_1,\ldots,u_k].
		\]
		
		\If{the graph associated with $L$ already has connected components $C_1,\ldots,C_k$}{
			Set $M^\star=L$ and stop.
		}
		
		\For{$j=1,\ldots,k$}{
			Set
			\[
			L_j=L[C_j,C_j],
			\qquad
			n_j=|C_j|.
			\]
			
			\If{$n_j=1$}{
				Set
				\[
				M_j^\star=[0].
				\]
			}
			\Else{
				Choose $Q_j\in\mathbb R^{n_j\times(n_j-1)}$ whose columns form an orthonormal basis of $\mathbf 1_{n_j}^{\perp}$.\;
				
				Solve
				\[
				\begin{aligned}
					\min_{M_j}\quad
					&\|M_j-L_j\|_F^2\\
					\text{s.t.}\quad
					&M_j\mathbf 1_{n_j}=0,\quad M_j=M_j^T,\\
					&(M_j)_{pq}\le0\quad(p\ne q),\\
					&Q_j^TM_jQ_j\succeq \varepsilon I_{n_j-1}.
				\end{aligned}
				\]
			}
		}
		
	\end{algorithm}

	\begin{algorithm}[H]
		\caption{Assembly and Verification of the Optimized Laplacian}
		\label{alg:assembly_verification}
		\DontPrintSemicolon
		\footnotesize
		
		\KwData{Optimized blocks $M_1^\star,\ldots,M_k^\star$, prescribed partition $C_1,\ldots,C_k$, tolerance $\tau>0$.}
		\KwResult{Optimized Laplacian $M^\star$ and optimized weight matrix $W^\star$.}
		
		Assemble $M^\star$ by placing the blocks $M_1^\star,\ldots,M_k^\star$ in the prescribed vertex positions.\;
		
		Set all cross-block entries of $M^\star$ equal to zero.\;
		
		Construct the optimized weight matrix by
		\[
		W^\star_{ij}=-M^\star_{ij},\qquad i\ne j,
		\]
		and
		\[
		W^\star_{ii}=0.
		\]
		
		Apply numerical cleanup:
		\[
		W^\star_{ij}=0
		\quad
		\text{whenever}
		\quad
		|W^\star_{ij}|<\tau.
		\]
		
		Construct the weighted graph associated with $W^\star$ and compute its connected components.\;
		
		Verify that the connected components are exactly
		\[
		C_1,\ldots,C_k.
		\]
		
		Compute the eigenvalues of $M^\star$ and verify that, whenever $k<n$,
		\[
		0=\lambda_1(M^\star)=\cdots=\lambda_k(M^\star)
		<
		\lambda_{k+1}(M^\star).
		\]
		
		Return $M^\star$ and $W^\star$.\;
		
	\end{algorithm}
	\FloatBarrier
	
	
\section{Numerical Experiments}

In this section, we illustrate the proposed nearest-Laplacian construction
through numerical examples. The computations were carried out in MATLAB using
CVX. For each example, the optimized Laplacian is denoted by $M^\star$, and
the associated optimized weight matrix is obtained from
\[
W^\star_{ij}=-M^\star_{ij},\qquad i\ne j,
\]
with
\[
W^\star_{ii}=0.
\]
For graph construction and visualization, entries of $W^\star$ whose absolute
values are below a prescribed tolerance are set to zero.

In the numerical experiments, the parameter $\varepsilon$ is chosen as a small
positive connectivity margin. Its role is to enforce strict internal
connectivity of each prescribed block. Larger values of $\varepsilon$ impose
stronger block connectivity but may increase the Frobenius distance
$\|M^\star-L\|_F$.\\

\textbf{Example 1: Faction-Consistent Reconstruction of the Sampson Monastery Network}

We now apply the proposed nearest-Laplacian construction to the Sampson
monastery data at time period 3 \cite{ArefEtAl2020, Sampson1968}. This data set consists of affect relations
among $18$ trainee monks. In the original data, each monk ranked the three
members he liked the most and the three members he liked the least. Thus,
the data are given as a directed signed weighted matrix
$S\in\mathbb{R}^{18\times 18}$, whose entries lie between $-3$ and $3$.
Since the present formulation is developed for undirected nonnegative
weighted graph Laplacians, we first extract the positive-affection part of
the data by setting
$$
A^+_{ij}=\max\{S_{ij},0\}.
$$
We then construct an undirected nonnegative weighted adjacency matrix by
symmetrization,
$$
W=\frac{A^+ +(A^+)^T}{2},
$$
and define the corresponding graph Laplacian by
$$
L=D-W,\qquad D=\operatorname{diag}(W\mathbf{1}).
$$
The prescribed components are chosen according to the standard social
factions reported for the Sampson monastery data:
$$
C_1=\{1,2,7,12,14,15,16\},\qquad
C_2=\{4,5,6,9,11\},
$$
$$
C_3=\{3,17,18\},\qquad
C_4=\{8,10,13\}.
$$
Here $C_1$ corresponds to the Young Turks, $C_2$ to the Loyal Opposition,
$C_3$ to the Outcasts, and $C_4$ to the Waverers. From these prescribed
sets, we form the block-indicator matrix
$$
U=[u_1,u_2,u_3,u_4],
$$
where $u_j$ is the indicator vector of $C_j$. The input to our algorithm is
therefore the Laplacian $L$, the prescribed partition
$C_1,C_2,C_3,C_4$, the block-indicator matrix $U$, and a positive
connectivity parameter $\varepsilon$. Solving the proposed convex
nearest-Laplacian problem produces an optimized weighted Laplacian
$M^\star$ that is nearest to $L$ in the Frobenius norm and whose associated
graph has exactly the prescribed four connected components. Consequently,
the value $\|M^\star-L\|_F$ measures the minimum Laplacian perturbation
required to enforce the prescribed faction structure on the weighted
positive-affection network.\\

	\label{fig:example1}

{
The original graph does not already have the prescribed faction-based
connected-component structure, and therefore the blockwise semidefinite
optimization is used. For this experiment, we choose
\[
\varepsilon=0.10.
\]
The optimized graph has exactly four connected components, and the prescribed
partition is verified exactly.

The computed Frobenius distance is
\[
\|M^\star-L\|_F=10.2717983464349238.
\]
Moreover, the first four eigenvalues of $M^\star$ are numerically zero, while
\[
\lambda_5(M^\star)=2.5782721444>0.
\]
Hence,
\[
\dim\ker(M^\star)=4,
\]
which confirms that the optimized graph has exactly the prescribed four
connected components.

The algebraic connectivity values of the four prescribed blocks are
\[
\lambda_2(M^\star_1)=2.5783,\qquad
\lambda_2(M^\star_2)=3.1957,
\]
\[
\lambda_2(M^\star_3)=4.0000,\qquad
\lambda_2(M^\star_4)=6.4226.
\]
Thus, all prescribed faction-based blocks are internally connected.

\begin{figure}[H]
    \centering
    \includegraphics[width=\textwidth]{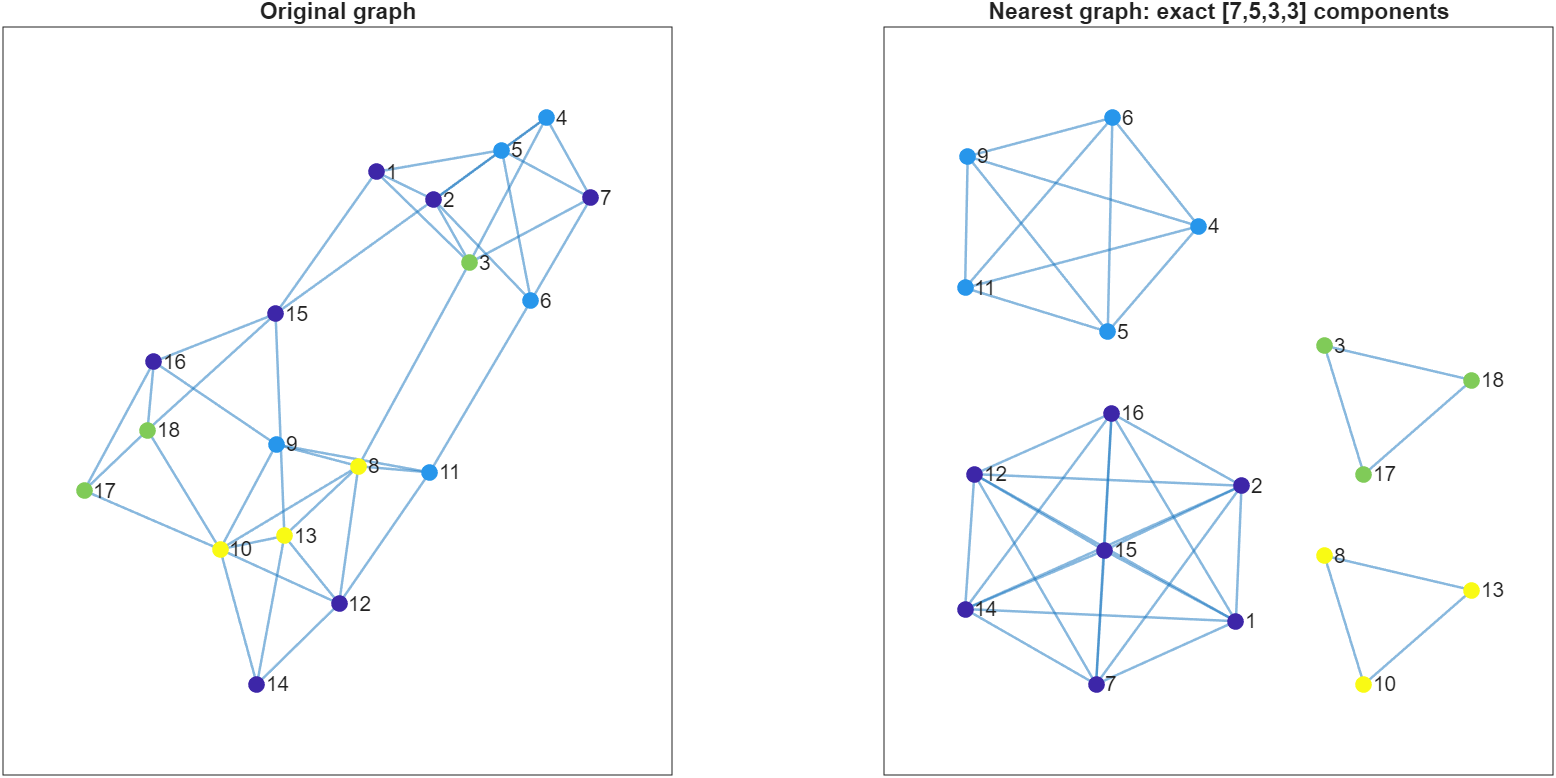}
    \caption{Original positive-affection network and its nearest faction-consistent weighted reconstruction. The optimized graph has exactly the prescribed Young Turks, Loyal Opposition, Outcasts, and Waverers factions as its connected components..}
    \label{fig:real18_faction}
\end{figure}

The numerical residuals confirm the validity of the computed solution. The
symmetry residual is zero to machine precision, the row-sum residual is
\[
7.511\times 10^{-11},
\]
the nullspace residual is
\[
\|M^\star U\|_F=1.061\times 10^{-10},
\]
and the positive off-diagonal violation is zero.

\begin{table}[H]
    \centering
    \caption{Numerical results for the faction-based real $18\times18$ data example.}
    \label{tab:real18_faction}
    \begin{tabular}{lc}
        \hline
        Quantity & Value \\
        \hline
        $\|M^\star-L\|_F$ & $10.2717983464349238$ \\
        Number of connected components & $4$ \\
        Exact prescribed partition verified & $1$ \\
        $\lambda_5(M^\star)$ & $2.5782721444$ \\
        $\lambda_2(M^\star_1)$ & $2.5783$ \\
        $\lambda_2(M^\star_2)$ & $3.1957$ \\
        $\lambda_2(M^\star_3)$ & $4.0000$ \\
        $\lambda_2(M^\star_4)$ & $6.4226$ \\
        $\|M^\star U\|_F$ & $1.061\times 10^{-10}$ \\
        Row-sum residual & $7.511\times 10^{-11}$ \\
        Positive off-diagonal violation & $0$ \\
        Minimum block LMI margin & $2.478$ \\
        Minimum $\lambda_2-\varepsilon$ & $2.478$ \\
        \hline
    \end{tabular}
\end{table}

In the faction-based real $18\times18$ example, the minimum distance is
\[
\|M^\star-L\|_F=10.2717983464349238.
\]
Hence, any other feasible Laplacian satisfying the same prescribed component
constraints and the same connectivity-margin condition is at least this far
from the original Laplacian, up to numerical solver tolerance.

The Sampson monastery experiment illustrates a socially interpretable application of the proposed nearest-Laplacian construction. The four faction sets are prescribed from the established sociological classification of the monks, whereas the observed positive-affection network does not possess these factions as its exact connected components. The optimized Laplacian $M^\star $ herefore represents the unique minimum-perturbation reconstruction of the observed network whose connected components coincide exactly with the Young Turks, Loyal Opposition, Outcasts, and Waverers. In particular, all positive cross-faction interactions are removed, while the within-faction weights are adjusted as little as possible subject to maintaining internal connectivity of every faction. Consequently, the distance $\||M^\star - L||_F$  provides a quantitative measure of the inconsistency between the observed positive-affection network and the prescribed faction structure, whereas the algebraic connectivities  $\lambda_2(M^\star_j)$ quantify the internal connectivity of the reconstructed faction subnetworks.

\textbf{Note:} The corresponding optimized Laplacian matrix $M^\star$ is displayed in block
form in Appendix~\ref{app:real18_faction_matrix}.
}\\

\textbf{Example 2}: Prescribed component sizes \texorpdfstring{$[3,2,2,1]$}{[3,2,2,1]}

We consider a graph with $n=8$ vertices. The original graph is connected, and
we prescribe the following four connected components:
\[
C_1=\{1,2,3\},\qquad
C_2=\{4,5\},\qquad
C_3=\{6,7\},\qquad
C_4=\{8\}.
\]
Thus,
\[
k=4,
\]
and the prescribed component sizes are
\[
[3,2,2,1].
\]

Let
\[
U=[u_1,u_2,u_3,u_4]
\]
be the corresponding block-indicator matrix. Solving the strengthened
nearest-Laplacian problem produces an optimized Laplacian matrix $M^\star$
whose associated graph has exactly the prescribed four connected components.

In this example, the computed Frobenius distance is
\[
\|M^\star-L\|_F = 6.1052091693.
\]
The optimized graph has exactly four connected components, and the prescribed
partition is verified exactly. Moreover, the first four eigenvalues of
$M^\star$ are numerically zero, while
\[
\lambda_5(M^\star)=1.8229387728>0.
\]
Hence,
\[
\dim\ker(M^\star)=4,
\]
which confirms that the optimized graph has exactly the prescribed four
connected components.

The algebraic connectivity values of the non-singleton prescribed blocks are
given by
\[
\lambda_2(M^\star_1)=1.8229,\qquad
\lambda_2(M^\star_2)=3.9733,\qquad
\lambda_2(M^\star_3)=3.4578.
\]
Thus, each non-singleton prescribed block is internally connected, while the
singleton block $C_4=\{8\}$ is connected by convention.

Figure~\ref{fig:example1} shows the original graph and the optimized nearest
graph with exactly the prescribed component structure.

\begin{figure}[H]
	\centering
	\includegraphics[width=\textwidth]{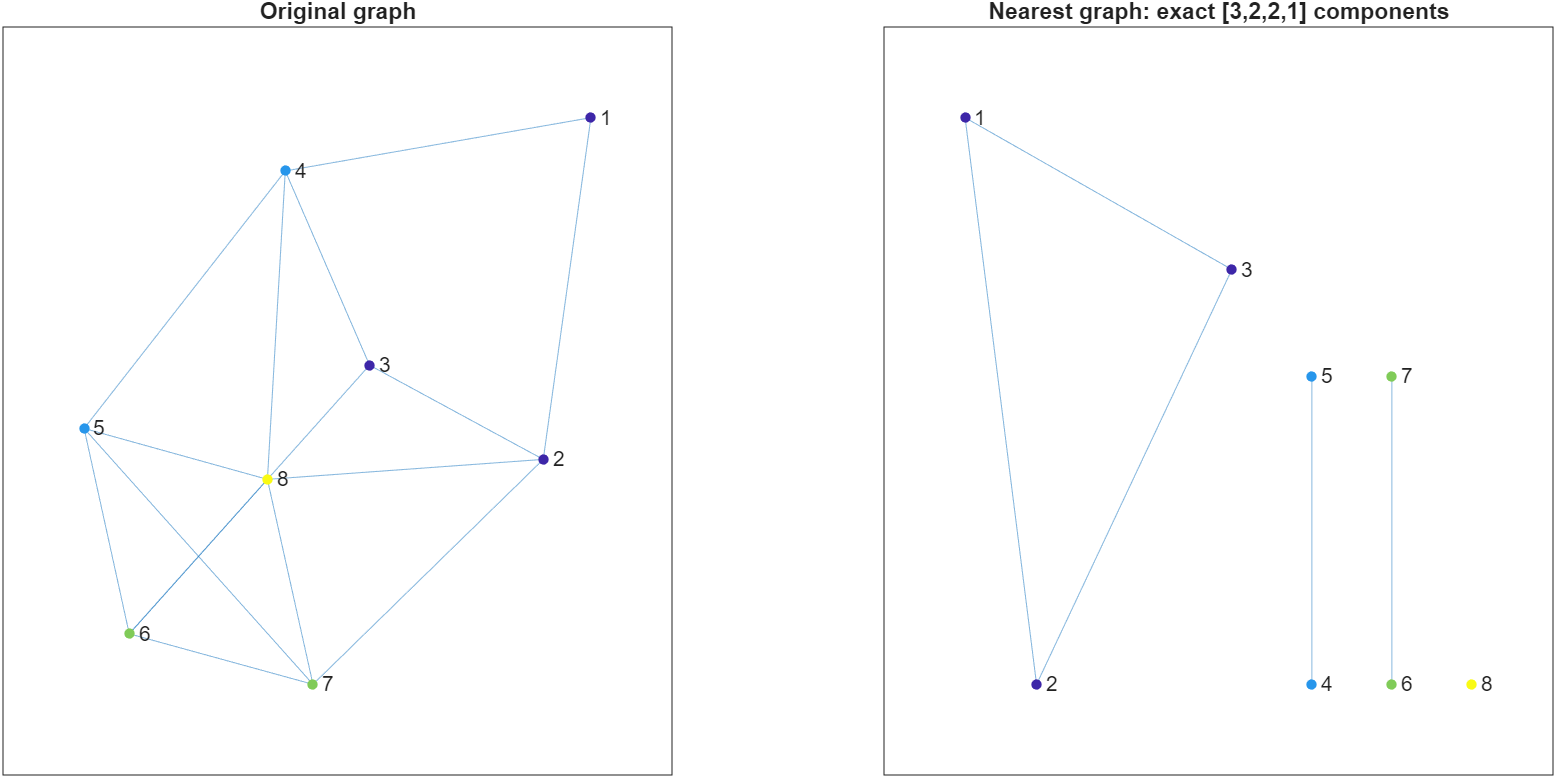}
	\caption{Original graph and optimized nearest graph with exactly prescribed component sizes $[3,2,2,1]$.}
	\label{fig:example1}
\end{figure}

For completeness, we also report the main numerical diagnostics. The symmetry
residual is zero to machine precision, the row-sum residual is
\[
3.253\times 10^{-12},
\]
the nullspace residual is
\[
\|M^\star U\|_F = 3.325\times 10^{-12},
\]
and the positive off-diagonal violation is zero. These values confirm the
numerical validity of the computed solution.

\begin{table}[H]
	\centering
	\caption{Numerical results for Example 2.}
	\label{tab:example1}
	\begin{tabular}{lc}
		\hline
		Quantity & Value \\
		\hline
		$\|M^\star-L\|_F$ & $6.1052091693$ \\
		Number of connected components & $4$ \\
		Exact prescribed partition verified & $1$ \\
		$\lambda_5(M^\star)$ & $1.8229387728$ \\
		$\lambda_2(M^\star_1)$ & $1.8229$ \\
		$\lambda_2(M^\star_2)$ & $3.9733$ \\
		$\lambda_2(M^\star_3)$ & $3.4578$ \\
		$\|M^\star U\|_F$ & $3.325\times 10^{-12}$ \\
		Row-sum residual & $3.253\times 10^{-12}$ \\
		Positive off-diagonal violation & $0$ \\
		\hline
	\end{tabular}
\end{table}

\section{Relation to Graph Partitioning and Computational Implications}

Classical graph partitioning problems, such as minimum cut, balanced
$k$-partition, and normalized cut, often involve optimizing over unknown
vertex assignments. Such formulations contain discrete or combinatorial
constraints, and many of these problems are NP-hard.

The problem considered in this paper is different. Here, the desired partition
\[
C_1,\ldots,C_k
\]
is prescribed in advance. We do not optimize over the assignment of vertices to
components. Instead, we optimize over the continuous entries of a Laplacian
matrix $M$ while enforcing the prescribed component structure through
nullspace and block-connectivity constraints.

The strengthened nearest-Laplacian problem is
\[
\begin{aligned}
	\min_{M\in\mathbb R^{n\times n}}\quad
	&\|M-L\|_F^2\\
	\text{subject to}\quad
	&MU=0,\\
	&M=M^T,\\
	&M_{ij}\le0,\qquad i\ne j,\\
	&M_j\succeq\varepsilon P_j,\qquad j=1,\ldots,k.
\end{aligned}
\]
The objective function is a strictly convex quadratic function of $M$. The
constraints $MU=0$ and $M=M^T$ are affine, the inequalities $M_{ij}\le0$ are
linear, and the constraints
\[
M_j-\varepsilon P_j\succeq0
\]
are linear matrix inequalities. Therefore, the problem is a convex
semidefinite optimization problem.

Thus, the proposed formulation avoids the combinatorial search present in
classical graph partitioning methods. Standard convex optimization solvers can
compute the unique global minimizer to numerical precision.
	
	\section{Conclusion}
	
	In this paper, we studied the problem of constructing the nearest graph
	Laplacian matrix with a prescribed connected-component structure. Given an
	original Laplacian matrix $L$ and a prescribed partition
	\[
	V=C_1\cup C_2\cup\cdots\cup C_k,
	\qquad
	C_i\cap C_j=\emptyset\quad (i\ne j),
	\]
	we formulated a convex optimization problem for finding a Laplacian matrix
	$M^\star$ nearest to $L$ in the Frobenius norm.
	
	The prescribed component structure is encoded through the block-indicator
	matrix
	\[
	U=[u_1,u_2,\ldots,u_k].
	\]
	The constraint
	\[
	MU=0
	\]
	ensures that the prescribed block-indicator vectors belong to the nullspace of
	the optimized Laplacian. This guarantees that the resulting graph has at least
	$k$ connected components. To ensure exactly the prescribed components, we
	introduced additional block-connectivity constraints
	\[
	M_j\succeq \varepsilon P_j,\qquad j=1,\ldots,k,
	\]
	where $M_j$ is the principal block corresponding to $C_j$ and $P_j$ is the
	orthogonal projector onto $\mathbf 1_{n_j}^{\perp}$.
	
	We proved that the strengthened formulation is a convex semidefinite
	optimization problem with a unique global minimizer. Moreover, the optimized
	Laplacian $M^\star$ has exactly the prescribed connected components
	\[
	C_1,\ldots,C_k,
	\]
	and its nullspace is given by
	\[
	\ker(M^\star)=\operatorname{span}\{u_1,u_2,\ldots,u_k\}.
	\]
	For numerical implementation, we used the equivalent reduced constraint
	\[
	Q_j^T M_j Q_j\succeq \varepsilon I_{n_j-1},
	\]
	which enforces positive definiteness on the subspace orthogonal to
	$\mathbf 1_{n_j}$ and avoids the unavoidable zero direction present in the
	projector formulation.
	
	The numerical examples illustrate that the proposed method successfully
	constructs the nearest Laplacian with exactly the desired component structure,
	including both consecutive and nonconsecutive prescribed components. 

\section*{Appendix}

\textbf{{Optimized Laplacian Matrix for the Faction-Based Real \texorpdfstring{$18\times18$}{18x18} Example}}
\label{app:real18_faction_matrix}

For the faction-based real $18\times18$ example, the prescribed partition is
\[
C_1=\{1,2,7,12,14,15,16\},
\qquad
C_2=\{4,5,6,9,11\},
\]
\[
C_3=\{3,17,18\},
\qquad
C_4=\{8,10,13\}.
\]
After ordering the vertices according to
\[
C_1,\ C_2,\ C_3,\ C_4,
\]
the optimized Laplacian matrix $M^\star$ has the block diagonal form
\[
M^\star=
\begin{bmatrix}
M^\star_1 & 0 & 0 & 0\\
0 & M^\star_2 & 0 & 0\\
0 & 0 & M^\star_3 & 0\\
0 & 0 & 0 & M^\star_4
\end{bmatrix}.
\]
The blocks of $M^\star$ are given below. The entries are rounded to four
decimal places.

\[
M^\star_1=
\begin{bmatrix}
 3.8571 & -1.4286 & -0.4286 & -0.4286 & -0.1429 & -1.1429 & -0.2857\\
-1.4286 &  5.2857 & -0.7143 & -0.7143 & -0.4286 & -1.4286 & -0.5714\\
-0.4286 & -0.7143 &  3.2857 & -0.7143 & -0.4286 & -0.4286 & -0.5714\\
-0.4286 & -0.7143 & -0.7143 &  4.2857 & -1.4286 & -0.4286 & -0.5714\\
-0.1429 & -0.4286 & -0.4286 & -1.4286 &  2.8571 & -0.1429 & -0.2857\\
-1.1429 & -1.4286 & -0.4286 & -0.4286 & -0.1429 &  4.8571 & -1.2857\\
-0.2857 & -0.5714 & -0.5714 & -0.5714 & -0.2857 & -1.2857 &  3.5714
\end{bmatrix}.
\]

\[
M^\star_2=
\begin{bmatrix}
 3.3200 & -1.6800 & -0.4800 & -0.6800 & -0.4800\\
-1.6800 &  4.3200 & -1.4800 & -0.6800 & -0.4800\\
-0.4800 & -1.4800 &  3.7200 & -0.4800 & -1.2800\\
-0.6800 & -0.6800 & -0.4800 &  3.3200 & -1.4800\\
-0.4800 & -0.4800 & -1.2800 & -1.4800 &  3.7200
\end{bmatrix}.
\]

\[
M^\star_3=
\begin{bmatrix}
 2.7778 & -1.2222 & -1.5556\\
-1.2222 &  2.7778 & -1.5556\\
-1.5556 & -1.5556 &  3.1111
\end{bmatrix}.
\]

\[
M^\star_4=
\begin{bmatrix}
 4.6667 & -2.6667 & -2.0000\\
-2.6667 &  5.0000 & -2.3333\\
-2.0000 & -2.3333 &  4.3333
\end{bmatrix}.
\]

The zero off-diagonal blocks show that the optimized graph has no edges between
distinct prescribed factions. Each diagonal block is a Laplacian block with
zero row sum and nonpositive off-diagonal entries.

	\section*{AI-Assistance Disclosure}
	
	The authors used an artificial intelligence language model for rephrasing,
	language polishing, and basic grammatical correction of certain sentences. All
	mathematical results, algorithms, proofs, and scientific conclusions presented
	in this manuscript are entirely authored, verified, and approved by the authors.
	
	\section*{Declaration of Competing Interest}
	
	The authors declare that they have no known competing financial interests or
	personal relationships that could have appeared to influence the work reported
	in this paper.
	



\end{document}